\documentclass[11pt]{article}

\usepackage{amsmath,amsthm,amsfonts,amssymb,amscd,mathrsfs}
\usepackage{hyperref}
\hypersetup{breaklinks=true,
pagecolor=white,
colorlinks=true,citecolor=black,filecolor=black,linkcolor=black,urlcolor=black
}
\theoremstyle{plain} 
\newtheorem{Thm}{Theorem}
\newtheorem{Lem}{Lemma}
\newtheorem{Def}{Definition}
\newtheorem{Cor}{Corollary}
\newtheorem{proposition}{Proposition}

\theoremstyle{definition}

\def\re{\mathbb R} 
\def\lV{\left\Vert }
\def\rV{\right\Vert}
\def\lv{\left\vert }
\def\rv{\right\vert}

\begin{document}
\date{November 12, 2014}

\title{Order preserving and order reversing operators on the class
of convex functions in Banach spaces}

\author{Alfredo N. Iusem\footnote{Instituto de Matem\'atica Pura e Aplicada 
(IMPA), Estrada Dona
Castorina 110, Rio de Janeiro, RJ, CEP 22460-320, Brazil,
{\tt iusp@impa.br}.
The work of this author author was partially supported by CNPq grant no. 301280/86.}
\qquad Daniel Reem\footnote{Instituto de Ci\^encias Matem\'aticas e de Computa\c c\~ao 
(ICMC), University of S\~ao Paulo at S\~ao Carlos, Avenida Trabalhador S\~ao-carlense 400, Centro,
CEP 13566-590, S\~ao Carlos, SP, Brazil,
{\tt dream@icmc.usp.br}.
Part of the work of this author was done while he was at the Instituto de Matem\'atica Pura e Aplicada (IMPA), Rio de Janeiro, Brazil, and the work was partially supported by an IMPA ``P\'os-doutorado de Excel\^encia" postdoctoral scholarship.}
\qquad Benar F. Svaiter\footnote{Instituto de Matem\'atica Pura e Aplicada, 
(IMPA), Estrada Dona
Castorina 110, Rio de Janeiro, RJ, CEP 22460-320, Brazil,
{\tt benar@impa.br}.
The work of this author author was partially supported by
CNPq grants no. 
474944/2010-7, 303583/2008-8 and  FAPERJ grant E-26/110.821/2008.
}
}

\maketitle

\begin{abstract}

A remarkable result by
S. Artstein-Avidan 
and V. Milman 
states that, up to pre-composition with
affine operators, addition
of affine functionals, and multiplication by positive scalars, the only {\it fully order
preserving} mapping acting on the class of lower semicontinuous proper convex functions
defined on $\re^n$ 
is the identity operator, and the only {\it fully order reversing} one 
acting on the same set 
is the Fenchel conjugation. Here {\it fully order preserving (reversing) mappings}
are understood to be those which preserve (reverse) the pointwise order among convex
functions, are invertible, and such that their inverses also preserve (reverse)
such order.   
In this paper we establish a suitable extension of 
these results to  
order preserving and order reversing operators acting on the class of lower
semicontinous proper convex functions defined on arbitrary
infinite dimensional Banach spaces.

\end{abstract}

\bigskip  
\bigskip 


\bigskip  
\bigskip 


\newpage

\section{Introduction}\label{s1}

In their remarkable paper \cite{AM2}, 
S. Artstein-Avidan 
and V. Milman 
established
an interesting result on certain mappings acting on the class of convex functions defined
on $\re^n$. More precisely, let $\mathscr{C}(\re^n)$ be the set of lower semicontinuous
proper convex functions $f:\re^n\to\re\cup\{+\infty\}$. For $f,g\in\mathscr{C}(\re^n)$,
we write $f\le g$ if and only if $f(x)\le g(x)$ for all $x\in\re^n$. An operator 
$T:\mathscr{C}(\re^n)\to\mathscr{C}(\re^n)$ is 
said to be {\it order preserving} 
if $T(f)\le
T(g)$ whenever $f\le g$,
and {\it order reversing} 
if $T(f)\ge
T(g)$ whenever $f\le g$. We will say that $T$ is fully order preserving (reversing) if,
in addition, $T$ is invertible and its inverse $T^{-1}$ is also order preserving (reversing).
The main results in \cite{AM2} state that, up to pre-composition with
affine operators, addition
of affine functionals, and multiplication by positive scalars, the only fully order
preserving mapping acting on $\mathscr{C}(\re^n)$ 
is the identity operator, and the only fully order reversing one 
acting on the same set 
is the Fenchel conjugation (also called Legendre transform, or convex conjugation), 
meaning the operator $T$ given by $T(f)=f^*$ for
all $f\in\mathscr{C}(\re^n)$, where $f^*$ denotes the Fenchel conjugate
of $f$, defined as $f^*(u)=\sup_{x\in\re^n}\{\langle u,x\rangle -f(x)\}$
for all $u\in\re^n$. 

We mention that order preserving isomorphisms acting on sets with specific structures
had been studied for some time, in some cases in connection with applications to
Physics (see. e.g., \cite{Mol}, \cite{NoS}, \cite{Rot}, \cite{Scr}), but Artstein-Avidan
and Milman's work \cite{AM2} was the first to deal with this issue  in connection with
the space of convex functions.   

The results from \cite{AM2} were the starting point of several interesting
developments, consisting mainly of replacing $\mathscr{C}(X)$ by other sets. Along
this line, we mention for instance 
the characterizations of fully order reversing mappings 
acting on:
\begin{itemize}
\item[i)]  the class of
$s$-concave functions (see \cite{AM1}), 
\item[ii)] the family of nonnegative convex functions vanishing at $0$ 
(see \cite{AM3}),  
\item [iii)] closed and convex cones,
with the order given by the set inclusion, and the polarity
mapping playing the role of the Fenchel conjugation (see \cite{Sch}),  
\item[iv)] closed and convex sets 
(see \cite{BoS}, \cite{Slo}),
\item [v)] ellipsoids 
(see \cite{ArS}).
\end{itemize} 
Additional recent related results can be found also in \cite{AAF}, \cite{AM4}, \cite{AM5}
and \cite{MSS}.

All these papers deal exclusively
with the finite dimensional case. In \cite{Wri}, it is proved
that the composition of two order reversing bijections acting
on the class of convex functions defined on certain locally convex topological
vector spaces is the identity if and only in the same happens
with their restrictions to the subset of affine functions, 
but
no characterization of the kind established in \cite{AM2} is given.
To our
knowledge, the extension of the above described
results in \cite{AM2} to infinite dimensional spaces
remains as an open issue. Addressing it is the main goal
of this paper. More precisely, we will generalize the 
analysis of
order preserving and order reversing operators 
developed by Milman and
Artstein-Avidan 
to the setting
of arbitrary Banach spaces. As a corollary, we also obtain
a characterization of order preserving involutions (see Corollary
\ref{c7}).     

We mention now the main differences between \cite{AM2}
and our work. 
 
When dealing, as we do, with arbitrary Banach spaces, the analysis becomes 
substantially more involved, demanding different tools.
This situation occurs at several points in the analysis.
For instance, 
a key element of the proof consists of establishing
the fact that the restriction of the operators of interest
to the space of affine functions is itself affine (see
Propositions \ref{p4} and \ref{p6} for a precise statement).
In the infinite dimensional case, we need to show that such
restriction is also continuous (cf. Propositions \ref{pr:el}
and \ref{p7}), which is trivial 
in a finite dimensional setting.

We mention now another difference between our approach and the one in \cite{AM2}.
As Milman and Artstein-Avidan rightly point out in \cite{AM2}, one can establish the characterization 
for the order preserving case and obtain the one for the order reversing case as an easy
corollary, or make the full analysis for order reversing operators and then get the result for
order preseving ones as a corollary. In \cite{AM2} they chose the second option, while we
follow in this paper the first one. 
As a consequence, excepting for some basic results, (e.g. Proposition \ref{p1}), or 
the use of some classical tools (e.g. Lemma \ref{l1}), which appear in \cite{AM2} as well as
in our work, both prooflines are quite  different. 

We also obtain, as a by-product of our analysis, some results which are possibly interesting
on their own, like two characterizations of affine functions in general Banach spaces (see Proposition
\ref{p2} and Lemma \ref{l2}), which are fully unrelated to the contents of \cite{AM2}.

\section{The order preserving case}\label{s2}
 
We start by introducing some quite standard notation.
Let $X$ be a real Banach space of dimension at least $2$. 
We denote by $X^*$ and $X^{**}$  its topological dual and
bidual, respectively. We define as $\mathscr{C}(X)$ the set of lower semicontinuous proper
convex functions $f:X\to\re\cup\{+\infty\}$, where the lower semicontinuity is understood 
to hold with respect to the strong (or norm) topology in $X$. For $f\in \mathscr{C}(X)$, 
$f^*:X^*\to\re\cup\{+\infty\}$ will denote
its Fenchel conjugate, defined as $f^*(u)=\sup_{x\in X}\{\langle u,x\rangle - f(x)\}$
for all $u\in X^*$,
where $\langle\cdot,\cdot\rangle$ stands for the duality coupling, i.e. $\langle u,x\rangle
=u(x)$ for all $(u,x)\in X^*\times X$. As usual, $f^{**}:X^{**}\to\re\cup\{+\infty\}$  
will indicate the  biconjugate of $f$, i.e. $f^{**}=(f^*)^*$. Given a
linear and continuous operator
$C:X^*\to X^*$, $C^*:X^{**}\to X^{**}$ will denote the adjoint of $C$, defined by
the equation $\langle C^*a,u\rangle =\langle a,Cu\rangle$ for all $(a,u)\in
X^{**}\times X^*$. We consider the pointwise order in $\mathscr{C}(X)$, i.e., given
$f,g\in \mathscr{C}(X)$, we write $f\le g$ whenever $f(x)\le g(x)$ for all $x\in X$.
We denote as $\re_{++}$ the set of strictly positive real numbers.
  
The next definition introduces the family of operators on $\mathscr{C}(X)$
whose characterization is the main goal of this paper.

\begin{Def}\label{d1}
\begin{itemize}
\item[i)] An operator $T:\mathscr{C}(X)\to\mathscr{C}(X)$ is {\it order preserving} whenever
$f\le g$ implies $T(f)\le T(g)$.
\item[ii)] An operator $T:\mathscr{C}(X)\to\mathscr{C}(X)$ is {\it fully order preserving} whenever
$f\le g$ iff $T(f)\le T(g)$, and additionally $T$ is onto.
\item[iii)] We define as $\mathscr{B}$ the family of fully order preserving 
operators $T:\mathscr{C}(X)\to\mathscr{C}(X)$. 
\end{itemize}
\end{Def}

We will prove that $T$ belongs to $\mathscr{B}$ if and only if there exist  
$c\in X$, $w\in X^*$, $\beta\in\re$, $\tau\in\re_{++}$ and a continuous automorphism 
$E$ of $X$ such that 
$$
T(f)(x) =\tau f(Ex+c) +\langle w,x\rangle+\beta,
$$
for all $f\in \mathscr{C}(X)$ and all $x\in X$ (see Theorem \ref{t1} below).

We mention that 
for technical reasons, we work in this paper under the assumption that
dim$(X)\ge 2$. We do not make this fact explicit in 
our main results (Theorems \ref{t1} and \ref{t2}, and Corollary \ref{c7}), because
for the case of dim$(X)=1$ (and indeed for any finite  dimensional $X$), all
of them 
have been
established in \cite{AM2}. 

We begin with an elementary property of the operators in $\mathscr{B}$.

\begin{proposition}\label{p1} If $T$ belongs to $\mathscr{B}$ then,
\begin{itemize}
\item[i)] $T$ is one-to-one,
\item[ii)] given a family $\{f_i\}_{i\in I}\subset \mathscr{C}(X)$ such that $\sup_{i\in I}f_i$
belongs to $\mathscr{C}(X)$ (i.e., the supremum is not identically 
equal to $+\infty$), it holds that
$T(\sup_{i\in I}f_i)=\sup_{i\in I}T(f_i)$.
\end{itemize}
\end{proposition}

\begin{proof} 
\begin{itemize}
\item[i)] If $T(f)=T(g)$ then $T(f)\le T(g)$ and $T(g)\le T(f)$, so that, since
$T$ is fully order preserving,  $f\le g$ and $g\le f$, implying that $f=g$.
\item[ii)] Let $\widehat f = \sup_{i\in I}f_i$. Since $f_i\le \widehat f$ for all $i\in I$
and $T$ is order preserving, we get that $T(f_i)\le T\left(\widehat f\right)$ for all $i\in I$,
so that
\begin{equation}\label{e1}
\sup_{i\in I}T(f_i)\le T\left(\widehat f\right)=T\left(\sup_{i\in I}f_i\right).
\end{equation}
In view of \eqref{e1}, $\sup_{i\in I}T(f_i)$ belongs to $\mathscr{C}(X)$. Since $T$ is
onto, there exists $g\in \mathscr{C}(X)$ such that $\sup_{i\in I} T(f_i)=T(g)$. Observe
that $T(f_i)\le\sup_{i\in I}T(f_i)=T(g)$, so that, since $T$ is fully order preserving,
$f_i\le g$ for all $i\in I$, and hence $\widehat f=\sup_{i\in I} f_i\le g$.
Since $T$ is order preserving,	
\begin{equation}\label{e2}
T\left(\sup_{i\in I}f_i\right)=T\left(\widehat f\right)\le T(g)=\sup_{i\in I}T(f_i).
\end{equation}
The result follows from \eqref{e1} and \eqref{e2}.
\end{itemize}
\end{proof}

An important consequence of Proposition \ref{p1}(ii) is that an
operator in $\mathscr{B}$ is fully determined by its action on the
family of affine functions on $X$, because any element of $\mathscr{C}(X)$
is indeed a supremum of affine functions (see, e.g., \cite{VaT}, p. 90). 
Hence, we will analyze next
the behavior of the restrictions of operators in $\mathscr{B}$ to the
family of affine functions.

Define $\mathscr{A}\subset\mathscr{C}(X)$ as $\mathscr{A}=\{h:X\to\re {\rm \,\,\,such\,\,\,that\,\,\,} h 
{\rm \,\,\,is\,\,\, affine\,\,\, and\,\,\, continuous}\}$.
By definition, any $h\in\mathscr{A}$ is of the form 
$h(x)=\langle u,x\rangle+\alpha$ with
$(u,\alpha)\in X^*\times\re$. From now on such $h$ will be denoted as 
$h_{u,\alpha}$.
For $f\in \mathscr{C}(X)$, define 
$A(f)\subset\mathscr{A}$ as $A(f)=\{h\in\mathscr{A}: h\le f\}$.

We continue with an elementary characterization of affine functions.

\begin{proposition}\label{p2}
For all $f\in \mathscr{C}(X)$, $f$ belongs to $\mathscr{A}$ if and only if there exists a unique $u\in X^*$
such that $h_{u,\alpha}\in A(f)$ for some $\alpha\in\re$.
\end{proposition}

\begin{proof}
The ``only if" statement is elementary (an affine function 
cannot be majorized by another one with a different slope). 

We prove now the ``if" statement. A well known result in convex analysis
(see, e.g., \cite{VaT}, p. 90), establishes
that $f=\sup_{h\in A(f)}h$ for all $f\in \mathscr{C}(X)$, 
so that $A(f)\neq\emptyset$ for all $f\in \mathscr{C}(X)$.
If there exists a unique $u\in X^*$ such that $h_{u,\alpha}$ belongs to $A(f)$, then 
$A(f)=\{h_{u,\alpha_i}\}_{i\in I}$ for some set $I\neq\emptyset$, and hence, 
\begin{equation}\label{e4}
f(x)=\sup_{h\in A(f)}h(x)=\sup_{i\in I}h_{u,\alpha_i}(x)=\sup_{i\in I}\{\langle u,x\rangle+\alpha_i\}=
\langle u,x\rangle+\sup_{i\in I}\{\alpha_i\}
\end{equation}
for all $x\in X$. Since $f$ is proper, $\sup_{i\in I}\{\alpha_i\}<\infty$. Taking $\alpha=\sup_{i\in I}\{\alpha_i\}$,
we conclude from \eqref{e4} that $f(x)=\langle u,x\rangle+\alpha$, i.e., $f$ belongs to $\mathscr{A}$.
\end{proof}

\begin{Cor}\label{c1}
If $f=h_{u,\alpha}\in\mathscr{A}$ then $A(f)=\{h_{u,\delta}:\delta\le\alpha\}$.
\end{Cor}

\begin{proof}
Immediate from Proposition \ref{p2}.
\end{proof}

\begin{Cor}\label{c2}
Consider $h_{u,\alpha}\in \mathscr{A}$, $f\in \mathscr{C}(X)$. If $f\le h_{u,\alpha}$, then
$f=h_{u,\delta}$ with $\delta\le\alpha$.
\end{Cor}

\begin{proof} Take $h_{u',\alpha'}\in A(f)$. Then $h_{u',\alpha'}\le f\le h_{u,\alpha}$,
i.e. $h_{u',\alpha'}\in A\left(h_{u,\alpha}\right)$. By Corollary \ref{c1}, $u'=u$, i.e.,
there exists a unique $u$ such that $h_{u,\alpha'}\in A(f)$ for some $\alpha'\in\re$. 
By Proposition \ref{p2},
$f$  is affine, and the result follows from Corollary \ref{c1}.
\end{proof}

Next, we prove that operators in $\mathscr{B}$ map affine functions to affine functions.

\begin{proposition}\label{p3}
If $T$ belongs to $\mathscr{B}$ then 
\begin{itemize}
\item[i)] $T(h)\in\mathscr{A}$ for all $h\in \mathscr{A}$,
\item[ii)] if $T(f)\in\mathscr{A}$ for some $f\in \mathscr{C}(X)$ then $f\in \mathscr{A}$.
\end{itemize}
\end{proposition}

\begin{proof} 
\begin{itemize}
\item[i)]
Let $h=h_{u,\alpha}$. We will use Proposition \ref{p2} in order to  
establish affineness of $T(h)$. Take $h_{z,\delta}, h_{z',\delta'} \in A(T(h))$,
i.e. $h_{z,\delta}\le T(h), h_{z',\delta'}\le T(h)$. Since $T$ is onto, there
exist $g,g'\in\mathscr{C}(X)$ such that $h_{z,\delta}=T(g), h_{z',\delta'}=T(g')$,
so that $T(g)\le T(h), T(g')\le T(h)$. Since $T$ is fully order preserving,
we get $g\le h, g'\le h$. By Corollary \ref{c2}, there exist $\eta,\eta'\in\re$
such that $g=h_{u,\eta}, g'=h_{u,\eta'}$. Assume, without loss of generality, 
that $\eta\le\eta'$ so that
$g\le g'$, and hence, since  $T$ is order preserving,
\begin{equation}\label{e5}
h_{z,\delta}=T(g)\le T(g')=h_{z',\delta'}.
\end{equation}
In view of Corollary \ref{c2}, we get from \eqref{e5} that $z=z'$. We have proved
that there exists a unique $z\in X^*$ such that $h_{z.\delta}\in A(T(h))$ for some
$\delta\in\re$, and hence $T(h)$ is affine by Proposition \ref{p2}. 

\item[ii)] Note that whenever $T\in \mathscr{B}$, its inverse, the operator $T^{-1}$, also
belongs to $\mathscr{B}$. It follows that $T^{-1}$ also maps affine functions to affine
functions, which establishes the result.
\end{itemize}
\end{proof}

We have seen that operators  $T\in\mathscr{B}$ map $\mathscr{A}$ to $\mathscr{A}$. For
$T\in \mathscr{B}$, we denote as $\widehat T:\mathscr{A}\to\mathscr{A}$ the restriction of
$T$ to $\mathscr{A}$. Note that $\widehat T$ inherits from $T$ the properties of
being onto and fully order preserving in $\mathscr{A}$. However, if we start
with an operator in $\mathscr{A}$ with these properties, and we ``lift it up"
to $\mathscr{C}(X)$, the resulting operator may not belong to $\mathscr{B}$. 
This is due to the fact that the order preserving property 
is rather weak in $\mathscr{A}$ because, as we have seen, a pair of affine functions
is ordered only when they have the same linear part. 
More precisely, consider an operator $\widehat R:\mathscr{A}\to\mathscr{A}$
which is onto and fully order preserving. We can extend it to $\mathscr{C}(X)$ in
a natural way, defining $R:\mathscr{C}(X)\to\mathscr{C}(X)$ as $R(f)=\sup_{h\in A(f)}\widehat R(h)$.
Since $f\le g$ implies $A(f)\subset A(g)$, it is easy to prove that $R$ is
order preserving, but it might fail to be fully order preserving, as the
following example shows:

Take $X=\re$ and define $\widehat R:\mathscr{A}\to\mathscr{A}$ as
$$
\widehat R(h_{u,\alpha})=\begin{cases} h_{u,\alpha} \qquad\qquad\,\,\, {\rm if}\,\,\, u\in (-1,1)\\
h_{-u,\alpha}\qquad\qquad{\rm otherwise.}
\end{cases}
$$
It is immediate that $\widehat R$ is onto and fully order preserving; in fact it is
an involution, meaning that $\widehat R\left(\widehat R(h)\right)=h$ for all $h\in\mathscr{A}$. Consider
its extension $R:\mathscr{C}(X)\to \mathscr{C}(X)$ as defined above, 
and take $f_1(x)=1/2\lv x\rv$, $f_2(x)=
\max\{x,0\}$. It is easy to check that $R(f_1)=f_1=(1/2)\lv x\rv$, $R(f_2)=\lv x\rv$, so that
$R(f_1)\le R(f_2)$, but it is not true that $f_1\le f_2$ (the inequality fails
for $x<0$). In fact any operator $\widehat R:\mathscr{A}\to \mathscr{A}$ of the form
$\widehat R(h_{u,\alpha})=h_{\psi(u),\varphi(\alpha)}$ where $\psi:X^*\to X^*$
is a bijection and $\varphi:\re\to\re$ is increasing and onto, turns out to be
onto and fully order preserving in $\mathscr{A}$, while
operators $\widehat T:\mathscr{A}\to\mathscr{A}$ which are the restrictions of operators
$T$ in $\mathscr{B}$ have a much more specific form. 

We will prove next that
$\widehat T:\mathscr{A}\to\mathscr{A}$ is indeed affine for all $T\in \mathscr{B}$,
with the following meaning: 
we identify $\mathscr{A}$ with $X^*\times\re$, associating $h_{u,\alpha}$ to the
pair $(u,\alpha)$, and hence we write $\widehat T(u,\alpha)$ instead of $\widehat T(h_{u,\alpha})$,
and consider affineness of $\widehat T$ as an operator acting on $X^*\times\re$.

In view of Proposition \ref{p3}, with this notation, for $T\in \mathscr{B}$,
$\widehat T$ maps $X^*\times\re$ to $X^*\times\re$, and so we will write
\begin{equation}\label{e6}
\widehat T(u,\alpha)=(y(u,\alpha), \gamma(u,\alpha)),
\end{equation} 
with $y:X^*\times\re\to X^*, \gamma:X^*\times\re\to\re$, i.e. $y(u,\alpha)$ is the linear
part of $T(u,\alpha)$ and $\gamma(u,\alpha)$ is its additive constant, so that
$T(h_{u,\alpha})(x)=\langle y(u,\alpha),x\rangle+\gamma(u,\alpha)$.

We must prove that both $y$ and $\gamma$ are affine. We start with $y$, establishing
some of its properties in the following two propositions. 

\begin{proposition}\label{pp3} If $T\in\mathscr{B}$ then
$y(\cdot,\cdot)$, defined by \eqref{e6},
depends only upon its first argument.
\end{proposition}

\begin{proof}
Consider arbitrary pairs $(u,\alpha), (u,\delta)\in X^*\times\re$, and assume WLOG
that $\alpha\le\delta$, so that $h_{u,\alpha}\le h_{u,\delta}$, and, since
$T$ is order preserving,
\begin{equation}\label{e7}
h_{y(u,\alpha),\gamma(u,\alpha)}=T(h_{u,\alpha})\le T(h_{u,\delta}) =h_{y(u,\delta),\gamma(u,\delta)}.
\end{equation}
It follows from \eqref{e7} and Corollary \ref{c2} that $y(u,\alpha)=y(u,\delta)$ and hence
$y$ does no depend upon its second argument.
\end{proof}

In view of Proposition \ref{pp3}, we will
write in the sequel $y:X^*\to X^*$, with $y(u)=y(u,\alpha)$ for an 
arbitrary $\alpha\in\re$,
and also 
\begin{equation}\label{ee7}
\widehat T(u,\alpha)=(y(u),\gamma(u,\alpha)).
\end{equation}

We recall that $\phi:X^*\to\re$ is
{\it quasiconvex} if $\phi(su_1+(1-s)u_2)\le\max\{\phi(u_1),\phi(u_2)\}$ for all
$u_1,u_2\in X^*$ and all $s\in [0,1]$. 

\begin{proposition}\label{p4} Assume that $T$ belongs to $\mathscr{B}$. Then,
\begin{itemize}
\item[i)] $y:X^*\to X^*$, defined by \eqref{ee7}, is one-to-one and onto,
\item[ii)] both 
$\langle y(\cdot),x\rangle:X^*\to\re$ and
$\langle y^{-1}(\cdot),x\rangle:X^*\to\re$
are quasiconvex for all $x\in X$.
\end{itemize}
\end{proposition}

\begin{proof}
\quad i) Suppose that $y(u_1)=y(u_2)$ for some $u_1,u_2\in X^*$. Assume WLOG that
$\gamma(u_1,0)\le\gamma(u_2,0)$. Then $T(h_{u_1,0})\le T(h_{u_2,0})$, implying,
since $T$ is fully order preserving, that $h_{u_1,0}\le h_{u_2,0}$, and hence $u_1=u_2$
by Corollary \ref{c2}, so that $y$ is one-to-one. Surjectivity of $y$ follows from
surjectivity of $\widehat T$, which is a consequence  of Proposition \ref{p3}(ii).

\noindent ii) Take $u_1,u_2\in X^*, s\in[0,1]$. Note that
$$
h_{su_1+(1-s)u_2,0}=sh_{u_1,0}+(1-s)h_{u_2,0}\le\max\{h_{u_1,0},h_{u_2,0}\}.
$$
Hence,
\begin{equation}\label{e8}
T\left(h_{su_1+(1-s)h_2,0}\right)\le T\left(\max\left\{h_{u_1,0},h_{u_2,0}\right\}\right)=
\max\left\{T\left(h_{u_1,0}\right),T\left(h_{u_2,0}\right)\right\},
\end{equation}
using the order preserving property of $T$ in the inequality and Proposition \ref{p1}(ii)
in the equality. Evaluating the leftmost and rightmost expressions in \eqref{e8} at a
point of the form $tx$ with $x\in X,t\in\re_{++}$, we get
\begin{equation}\label{e9}
\langle y(su_1+(1-s)u_2),tx\rangle +\gamma(su_1+(1-s)u_2),0)\le
\max\{\langle y(u_1),t x\rangle+\gamma(u_1,0),\langle y(u_2),t x\rangle
+\gamma(u_2,0)\}.
\end{equation}
Dividing both sides of \eqref{e9} by $t$ and taking limits with $t\to\infty$, we get
$$
\langle y(su_1+(1-s)u_2),x\rangle\le\max\{\langle y(u_1),x\rangle ,\langle y(u_2),x\rangle\},
$$
establishing quasiconvexity of $\langle y(\cdot),x\rangle$ for all $x\in X$.
Quasiconvexity of $\langle y^{-1}(\cdot),x\rangle$ follows using the same argument
with $\widehat T^{-1}$ instead of $\widehat T$.
\end{proof}

Next we recall a well known result on finite dimensional affine geometry. 

\begin{Lem}\label{l1} 
Let $V,V'$ be finite dimensional real vector spaces, with {\rm dim}$(V)\ge 2$.
If $Q:V\to V'$ is invertible and maps segments to segments, then $Q$ is affine.
\end{Lem}
 
\begin{proof} The result appears, e.g., in \cite{Art}, see also Corollary 2 in \cite{ChP}.
A short proof can be found in Remark 6 of \cite{AM2}.
\end{proof}

The next corollary extends this result to our infinite dimensional setting.

\begin{Cor}\label{c3} Let $Z,Z'$ be arbitrary real vector spaces with {\rm dim}$(Z)\ge 2$.
If $Q:Z\to Z'$ is invertible and maps segments to segments, then $Q$ is affine.
\end{Cor}

\begin{proof} Take  vectors $z_1,z_2\in Z$  and $s\in [0,1]$. We must prove
that 
\begin{equation}\label {e10}
Q(sz_1+(1-s)z_2)=sQ(z_1)+(1-s)Q(z_2).
\end{equation} 
Assume first that $z_1, z_2$ are linearly independent.
Let $V\subset Z$ be the two-dimensional subspace spanned
by $z_1,z_2$, $V'\subset Z'$ the subspace spanned by $Q(0)$, $Q(z_1)$ and $Q(z_2)$,
and
$Q_{|V}$ 
the restriction of $Q$ to $V$. Clearly,   
$Q_{|V}$ also maps segments to segments and its image is $V'$,
so that \eqref{e10} follows by applying Lemma \ref{l1} to   
$Q_{|V}$.
 
If $z_1, z_2$ are  colinear, 
assume WLOG that $z_1\neq 0$, 
and replace $z_2$ by $\widetilde  z_2=z_2+z_3$,
with $z_3$ linearly independent of $z_1,z_2$.
Conclude, using the same argument, that \eqref{e10} holds with $\widetilde z_2$ instead
of $z_2$, and then take limits with $z_3\to 0$ in order to recover \eqref{e10} with $z_2$. 
\end{proof}

We continue with another result, possibly of some interest on its own.

\begin{Lem}\label{l2}
If a mapping $M:X^*\to X^*$ satisfies:
\begin{itemize}
\item[i)] $M$ is one-to-one and onto,
\item[ii)] both 
$\langle M(\cdot),x\rangle:X^*\to\re$ and
$\langle M^{-1}(\cdot),x\rangle:X^*\to\re$ 
are quasiconvex for all $x\in X$,
\end{itemize}
then $M$ is affine.
\end{Lem}

\begin{proof} Take $u_1,u_2\in X^*$, $s\in [0,1]$. We show next that $M(su_1+(1-s)u_2)$
belongs to the segment between $M(u_1)$ and $M(u_2)$.  Define $\bar u =su_1+(1-s)u_2$.
By assumption (ii), 
\begin{equation}\label{e11}
\langle M(\bar u),x\rangle\le\max\{\langle M(u_1),x\rangle,\langle M(u_2),x\rangle\}.
\end{equation}
Note that if $M(u_1)=M(u_2)$ then $u_1=u_2=\bar u$ because $M$ is one-to-one, 
and the result holds trivially.
Assume that $M(u_1)\neq M(u_2)$, and
consider the halfspaces $U=\{x\in X:\langle  M(u_1)-M(u_2),x\rangle \ge 0\}$,
$W=\{x\in X:\langle M(u_1)-M(\bar u),x\rangle\ge 0\}$. Note that if $x\in U$ then the
maximum in the right hand side of \eqref{e11} is attained in the first argument,
in which case 
$\langle M(\bar u),x\rangle\le\langle M(u_1),x\rangle$, i.e., $x$ belongs to $W$.
We
have shown that $U\subset W$. It is an easy consequence of the Convex
Separation Theorem
that $M(u_1)-M(u_2)$ and $M(u_1)-M(\bar u)$ belong to the same halfline, i.e., there
exists $\sigma >0$ such that $M(u_1)-M(\bar u)=\sigma(M(u_1)-M(u_2))$, or equivalently,
\begin{equation}\label{e12}
M(su_1+(1-s)u_2)=\sigma M(u_2)+(1-\sigma)M(u_1).
\end{equation}
Reversing the roles of $u_1, u_2$, we conclude, 
with the same argument, that there exists $\sigma' >0$ such that
$$
M(su_1+(1-s)u_2)=(1-\sigma') M(u_2)+\sigma'M(u_1),
$$
implying that $\sigma=1-\sigma'$ and therefore $\sigma\in [0,1]$. Thus,
\eqref{e12} shows that $M(su_1+(1-s)u_2)$ belongs to the segment between
$M(u_1)$ and $M(u_2)$, and therefore $M$ maps all points in the
segment between $u_1$ and $u_2$ to points between $M(u_1)$ and $M(u_2)$.
The fact that the image of the first segment fills the second one results
from the same argument applied to $M^{-1}$, which enjoys, by assumption, the same 
quasiconvexity property as $M$. We have shown that $M$ maps segments to segments,
and then affineness of $M$ follows from Corollary \ref{c3}.
\end{proof}

\begin{Cor}\label{c4} If $T\in \mathscr{B}$ then $y:X^*\to X^*$, defined by \eqref{ee7}, is affine.
\end{Cor}

\begin{proof} The result follows from Proposition \ref{p4} and Lemma \ref{l2}.
\end{proof}

We establish next affineness of $\gamma(\cdot,\cdot)$. We start with an elementary
property of $\gamma(u,\cdot)$.

\begin{proposition}\label{p5} If $T$ belongs to $\mathscr{B}$ then
$\gamma(u,\cdot):\re\to\re$ is strictly increasing, one-to-one and and onto
for all $u\in X^*$.
\end{proposition}

\begin{proof}
Since $T$ is a bijection, we get from Proposition \ref{p3}(ii) that $\widehat T$ is also
a bijection, and hence, in view of Propositions \ref{pp3} and \ref{p4}(i), 
the same holds for  $\gamma(u,\cdot)$. If $\alpha\le\delta$
then $h_{u,\alpha}\le h_{u,\delta}$, which implies, since $T$ is order
preserving, $h_{y(u),\gamma(u,\alpha)}\le h_{y(u),\gamma(u,\delta)}$, so that,
in view of Corollary  \ref{c2}, $\gamma(u,\alpha)\le\gamma(u,\delta)$. Since 
$\gamma(u,\cdot)$ is one-to-one, we conclude that it is strictly increasing.  
\end{proof}

Next we establish affineness of $\gamma(\cdot,\cdot)$.

\begin{proposition}\label{p6}
If $T\in\mathscr{B}$ then $\gamma:X^*\times\re\to\re$, defined by \eqref{ee7}, is affine.
\end{proposition}

\begin{proof}
We start by proving that $\gamma$ is convex. Take $u_1,u_2\in X^*$,  
$\alpha_1, \alpha_2\in\re$
and $s\in [0,1]$. We consider first the case in which $u_1\neq u_2$. 
Define $\bar u=su_1+(1-s)u_2, \bar\alpha=s\alpha_1+(1-s)\alpha_2$.
Observe that
\begin{equation}\label{e13}
sh_{u_1,\alpha_1}+(1-s)h_{u_2,\alpha_2}\le\max\left\{h_{u_1,\alpha_1},h_{u_2,\alpha_2}\right\}.
\end{equation}
As in the proof of Proposition \ref{p4}(ii), we get from \eqref{e13} that
$T\left(h_{\bar u,\bar\alpha}\right)\le\max\left\{T\left(h_{u_1,\alpha_1}\right),
T\left(h_{u_2,\alpha_2}\right)\right\}$,
so that, for all $x\in X$,
$$
\langle y(su_1+(1-s)u_2),x\rangle+\gamma(\bar u,\bar\alpha)=
s\langle y(u_1),x\rangle+(1-s)\langle y(u_2),x\rangle+\gamma(\bar u,\bar\alpha)
\le
$$
\begin{equation}\label{e14}
\max\{\langle y(u_1),x\rangle
+\gamma(u_1,\alpha_1),\langle y(u_2),x\rangle +\gamma(u_2,\alpha_2)\},
\end{equation}
using Corollary \ref{c4} in the equality. Since $u_1\neq u_2$, and $y$ is one-to-one
by Proposition \ref{p4}(i), we have that $y(u_1)\neq y(u_2)$, and hence there exists 
$\bar x\in X$ such that
$$
\langle y(u_1),\bar x\rangle+\gamma(u_1,\alpha_1)=
\langle y(u_2),\bar x\rangle+\gamma(u_2,\alpha_2),
$$
so that
\begin{equation}\label{e15}
\langle y(u_1)-y(u_2),\bar x\rangle=\gamma(u_2,\alpha_2)-\gamma(u_1,\alpha_1).
\end{equation}
Replacing \eqref{e15} in \eqref{e14}, and noting that for $x=\bar x$ there is a tie
between both arguments for the maximum in the rightmost expression of \eqref{e14},
we obtain, after some elementary algebra,
\begin{equation}\label{e16}
\langle y(u_2),\bar x\rangle +s\gamma(u_2,\alpha_2)-s\gamma(u_1,\alpha_1)+
\gamma(\bar u,\bar\alpha)\le\langle y(u_2),\bar x\rangle+\gamma(u_2,\alpha_2).
\end{equation}
Rearranging terms in \eqref{e16} and using the definitions of $\bar u, \bar\alpha$, we get
\begin{equation}\label{e17}
\gamma(s(u_1,\alpha_1)+(1-s)(u_2,\alpha_2))\le s\gamma(u_1,\alpha_1)+(1-s)\gamma(u_2,\alpha_2),
\end{equation}
establishing joint convexity of $\gamma$ in its two arguments. In order to obtain
affineness of $\gamma$, it suffices to prove that \eqref{e17} holds indeed with equality.
Suppose that, on the contrary, there exist $u_1,u_2\in X^*$, $u_1\neq u_2$, $\alpha_1,\alpha_2\in\re$ and
$\bar s\in [0,1]$ such that \eqref{e17} holds with strict inequality, and take $\theta\in\re$
such that
\begin{equation}\label{e18}
\gamma(\bar s(u_1,\alpha_1)+(1-\bar s)(u_2,\alpha_2))
<\theta<\bar s\gamma(u_1,\alpha_1)+(1-\bar s)\gamma(u_2,\alpha_2).
\end{equation}
Define now $\bar\gamma= 
\gamma(\bar s(u_1,\alpha_1)+(1-\bar s)(u_2,\alpha_2))$,
$\gamma_1=\gamma(u_1,\alpha_1)$, $\gamma_2=\gamma(u_2,\alpha_2)$.
Since $\gamma(\bar u,\cdot)$ is onto by Proposition \ref{p5}, there exists
$\eta\in\re$ such that $\theta=\gamma(\bar u,\eta)$, and hence it follows from
\eqref{e18} that
$$
h_{y(\bar u),\bar\gamma}< h_{y(\bar u),\gamma(\bar u,\eta)}<h_{y(\bar u),
\bar s\gamma_1+(1-\bar s)\gamma_2}\le\max\left\{h_{y(u_1),\gamma_1},
h_{y(u_2),\gamma_2}\right\}=
$$
\begin{equation}\label{e19}
\max\{\widehat T(u_1,\alpha_1),\widehat T(u_2,\alpha_2)\}=T\left(\max\left\{
h_{u_1,\alpha_1},h_{u_2,\alpha_2}\right\}\right).
\end{equation}
Since $h_{y(\bar u),\bar\gamma}=T\left(h_{\bar u,\bar\alpha}\right)$ and
$h_{y(\bar u),\gamma(\bar u,\eta)}=T\left(h_{\bar u,\eta}\right)$, we can rewrite \eqref{e19}
as
$$
T\left(h_{\bar u,\bar\alpha}\right)<T\left(h_{\bar u,\eta}\right)\le
T\left(\max\left\{h_{u_1,\alpha_1},h_{u_2,\alpha_2}\right\}\right),
$$
which gives, since $T$ is fully order preserving and one-to-one,
$$
h_{\bar u,\bar\alpha}<h_{\bar u,\eta}\le
\max\left\{h_{u_1,\alpha_1},h_{u_2,\alpha_2}\right\},
$$
and therefore, for all $x\in X$,
\begin{equation}\label{e20}
\langle\bar u,x\rangle+\bar\alpha <\langle\bar u,x\rangle+\eta\le\max\{
\langle u_1,x\rangle+\alpha_1,\langle u_2,x\rangle+\alpha_2\}.
\end{equation}
Since $u_1\neq u_2$, there exists $\bar x$ such that
 $\langle u_1-u_2,\bar x\rangle=\alpha_2-\alpha_1$,
in which case
$$
\langle\bar u,\bar x\rangle+\bar\alpha=\langle u_1,\bar x\rangle +\alpha_1=
\langle u_2,\bar x\rangle+\alpha_2,
$$
which contradicts \eqref{e20}, and establishes that 
\begin{equation}\label{e21}
\gamma(s(u_1,\alpha_1)+(1-s)(u_2,\alpha_2))= s\gamma(u_1,\alpha_1)+(1-s)\gamma(u_2,\alpha_2),
\end{equation}
completing the proof of affineness of $\gamma$ when $u_1\neq u_2$.
For the remaining case (i.e., $u_1=u_2$), 
we must prove affineness of $\gamma(u,\cdot)$ for a fixed $u$. 
Take now $u,v\in X^*$  with $v\neq 0$, so that $u\neq u+v$ and
\eqref{e21} holds with $u_1=u+v$, $u_2=u$. 
Note that 
$\gamma(\cdot,\alpha)$ is affine for any fixed $\alpha$, 
because for a fixed second argument 
\eqref{e21} holds trivially when $u_1=u_2$. 
So, the restrictions of $\gamma(\cdot,\bar\alpha)$ and $\gamma(\cdot,\alpha_1)$ 
to the subspace of $X$ spanned by $u$ and $v$, being affine, are continuous, and
taking limits with $v\to 0$ in \eqref{e21}, we get affineness of
$\gamma(u,\cdot)$ for all $u\in X^*$,
completing the proof.
\end{proof}

In the finite dimensional case, Proposition \ref{p6} would suffice for obtaining
an explicit form of $\widehat T$, but in our setting we still have to prove continuity
of $\widehat T$, which is not immediate from its affineness.
For proving continuity, we will need the following elementary result.

\begin{proposition}
\label{pr:el} Take $T\in \mathscr{B}$.
If $f\in\mathscr{C}(X)$ is finite everywhere, then $T(f)$
is also finite everywhere.
\end{proposition}

\begin{proof}
Consider $T\in\mathscr{B}$ and a finite everywhere $f\in \mathscr{C}(X)$. 
Suppose that $T(f)(x_0)=\infty$ for some $x_0\in X$ and define 
$g_0:X\to\re\cup\{\infty\}$ as
$$
g_0(x)=
\begin{cases}
0& {\rm if}\,\,\,x=x_0\\
\infty&\text{otherwise.}
\end{cases}
$$
Since $T$ is onto, there exists $f_0\in\mathscr{C}(X)$ such that $T(f_0)=g_0$. 
Since $f$ is finite
everywhere, $\widetilde f=\max\{f,f_0\}$ belongs to $\mathscr{C}(X)$. Therefore,
using Proposition \ref{p1}(ii),
$$
T\left(\widetilde f\right)=\max\{T(f),T(f_0)\}=\max\{T(f),g_0\},
$$ 
so that $T\left(\widetilde f\right)(x)=\infty$ for all $x\in X$.  Hence, we
have that $\widetilde f\in\mathscr{C}(X)$ and $T\left(\widetilde f\right)\notin\mathscr{C}(X)$, in
contradiction with our assumptions on $T$.
\end{proof}

\begin{proposition}\label{p7}
If $T\in\mathscr{B}$ then both $y:X^*\to X^*$ and $\gamma:X^*\times\re\to\re$ are continuous.
\end{proposition}

\begin{proof}
First we prove that $\gamma$ is continuous. Let $g=\lV \cdot\rV\in \mathscr{C}(X)$ and define 
$\bar g=T(g)$.  Let $B$ the the unit ball in $X^*$, i.e. $B=\{u\in X^*: \lV u\rV\le 1\}$. Since
$\lV x\rV=\sup_{u\in B}\left\{h_{u,0}(x)=\langle u,x\rangle\right\}$,
we have
\begin{equation}\label{ee22}
\bar g(x)=T(g)(x)=\sup_{u\in B}\left\{\left(T(h_{u,0})\right)(x)\right\}
=\sup_{u\in B}\{ h_{\widehat T (u,0)}(x)\}=
\sup_{u\in B}\{\langle y(u),x\rangle+\gamma(u,0)\}.
\end{equation}
Therefore,
\begin{equation}\label{ea22}
\bar g(0)=\sup_{u\in B}\{\gamma(u,0)\}.
\end{equation} 
Since $\gamma:X^*\times\re\to\re$ is affine by Proposition \ref{p6}, 
it can be written as $\gamma(u,\alpha)=\widetilde\gamma(u)+
\widehat\gamma(\alpha)+\mu$
where 
$\mu=\gamma(0,0)\in\re$  and 
$\widetilde\gamma:X^*\to\re,\widehat\gamma:\re\to\re$ are linear functionals.   
Using this representation of $\gamma$ and \eqref{ea22}, we conclude
that
$\sup_{u\in B}\{\widetilde\gamma(u)\}+\mu=\bar g(0)<\infty$,
where the inequality follows from Proposition \ref{pr:el} and the definition
of $g$. Therefore
$$
\lv\widetilde\gamma(u)\rv=\max\{\widetilde\gamma(u),\widetilde\gamma(-u)\}
\le\bar g(0)-\mu\in\re,
$$
for all $u\in B$, and hence
$\widetilde\gamma$ is a bounded linear functional with $\lV\widetilde\gamma\rV\le
\bar g(0)-\mu$, so that $\widetilde\gamma$ is continuous. 
Continuity of $\widehat\gamma$ follows trivially from its
linearity and the fact that its domain is finite dimensional. Hence
$\gamma$ is continuous.

For proving continuity of $y$, observe that $y$ can be written as $y(u)=Du+w$ 
where $D:X^*\to X^*$ is a linear operator 
and $w=y(0)\in X^*$. Combining this representation of $y$ with
\eqref{ee22} we have
\begin{equation}\label{eb22}
\sup_{u\in B}\{\langle Du+w,x\rangle+\langle\widetilde\gamma,u\rangle+\mu
\}=\bar g(x)<\infty
\end{equation}
for all $x\in X$, where the inequality follows from Proposition \ref{pr:el} and
the definition of $g$.
Take $u\in B$ and
conclude from \eqref{eb22} that
$$
\langle Du,x\rangle\le\bar g(x)-\langle w,x\rangle
-\widetilde\gamma(u)-\mu
\le \bar g(x)-\langle w,x\rangle
+\lV\widetilde\gamma\rV\,\lV u\rV-\mu
\le\bar g(x)-\langle w,x\rangle
+\lV\widetilde\gamma\rV-\mu.
$$
Therefore
$$
\sup_{u\in B}\{\lv\langle Du,x\rangle\rv\}=
\sup_{u\in B}\{\langle Du,x\rangle\}<\infty
$$
for all $x\in X$.
This means that the family of bounded linear operators
$\{Du\;|\;u\in B\}$ is pointwise bounded. Applying Banach-Steinhaus
uniform boundedness principle (see, e.g., \cite{Bre}, p. 32), 
we get that this family is uniformly
bounded, i.e., there exists $\nu<\infty$ such that
$\lV Du\rV\le\nu$ for all $u\in B$.
It follows that
$\lV D\rV\le\nu$, and therefore 
$D$ is bounded and linear, hence continuous. 
We conclude that the operator $y$ defined as $y(u)=Du+w$ is continuous, 
completing the proof.
\end{proof}

Now we present a more explicit formula for the operator $\widehat T$.

\begin{Cor}\label{c5} 
If $T\in\mathscr{B}$, then there exist 
$d\in X^{**}$, $w\in X^*$, $\beta\in\re$,  $\tau\in\re_{++}$ and a  continuous automorphism 
$D$ of $X^*$ such that $\widehat T(u,\alpha)=(Du+w,\langle d,u\rangle+\tau\alpha+\beta)$ for all
$u\in X^*$ and all $\alpha\in\re$.
\end{Cor}

\begin{proof} Being affine and continuous by virtue of Corollary \ref{c4} and Propositions  
\ref{p6} and \ref{p7}\, we have that $y(u)=Du+w$, 
$\gamma(u,\alpha)=\langle d,u\rangle+\tau\alpha+\beta$, with $D,d,w,\tau$ and $\beta$
as in the statement of the corollary. The facts that $D$
is an automorphism of $X^*$ and that $\tau$ is positive follow from Propositions
\ref{p4}(i) and \ref{p5}, which establish that $y$ is a bijection and
that $\gamma(u,\cdot)$ is increasing,
respectively.
\end{proof} 

We recall now two basic properties of the Fenchel biconjugate, which will be needed
in the proof of our main result.

\begin{proposition}\label{p8}
\begin{itemize}
\item[i)] For all $f\in \mathscr{C}(X), f^{**}_{|X}=f$, where $f^{**}_{|X}$ denotes the
restriction of $f^{**}$ to $X\subset X^{**}$.
\item[ii)] For all $f\in\mathscr{C}(X)$ and all $a\in X^{**}$, 
$f^{**}(a)=\sup_{h_{u,\alpha}\in A(f)}\{\langle a,u\rangle +\alpha\}$.
\end{itemize}
\end{proposition}

\begin{proof}
A proof of item (i), sometimes called Fenchel-Moreau Theorem, can be found, 
for instance, in
Theorem 1.11 of \cite{Bre}.
 
We proceed to prove item (ii).
Observe that, by definition of $f^{**}$, 
$f^{**}(a)=\sup_{u\in X^*}\{\langle a,u\rangle-f^*(u)\}$.
Also, by definition of $f^*$, $f^*(u)=\sup_{x\in X}\{\langle u,x\rangle - f(x)\}$,
so that $h_{u,-f^*(u)}(x)=\langle u,x\rangle -f^*(u)\le f(x)$ for all $x\in X$, and
therefore $h_{u,-f^*(u)}$ belongs to $A(f)$ for all $u\in X^*$ (we may assume $f^*(u)<+\infty$
when taking the supremum over $u\in X^*$, in the formula of $f^{**}(a)$).
Next, we prove that if $h_{u,\alpha}$ belongs to $A(f)$ then $\alpha\le -f^*(u)$.
Indeed, if $h_{u,\alpha}\in A(f)$ then, for all $x\in X$,
\begin{equation}\label{e26}
\alpha\le 
\inf_{x\in X}\{f(x)-\langle u,x\rangle\}=
\inf_{x\in X}\{-(\langle u,x\rangle - f(x))\}=
-\sup_{x\in X}\{\langle u,x\rangle - f(x)\}=
-f^*(u).
\end{equation}  
In view of \eqref{e26} and the fact that $h_{u,-f^*(u)}\in A(f)$ for all $u\in X^*$,
we have 
$$
\sup_{h_{u,\alpha}\in A(f)}\{\langle a,u\rangle +\alpha\}= 
\sup_{h_{u,-f^*(u)}\in A(f)}\{\langle a,u\rangle -f^*(u)\}= 
\sup_{u\in X^*}\{\langle a,u\rangle -f^*(u)\}= f^{**}(a).
$$
\end{proof}

We mention, parenthetically, that the Fenchel-Moreau Theorem can be easily deduced
from Proposition \ref{p8}(ii).
Next we present our first main result, characterizing the operators $T\in\mathscr{B}$

\begin{Thm}\label {t1}
An operator $T:\mathscr{C}(X)\to\mathscr{C}(X)$ is fully order preserving 
if and only if there exist  
$c\in X$, $w\in X^*$, $\beta\in\re$,  $\tau\in\re_{++}$ and a continuous automorphism 
$E$ of $X$ such that 
\begin{equation}\label{e23}
T(f)(x) =\tau f(Ex+c) +\langle w,x\rangle+\beta,
\end{equation}
for all $f\in\mathscr{C}(X)$ and all $x\in X$.
\end{Thm}

\begin{proof}
We start with the ``only if" statement.
As already mentioned, a basic convex analysis result establishes that $f=\sup_{h\in A(f)}h$.
In view of Propositions \ref{p1}(ii) and \ref{pp3}, we have
$$
T(f)=\sup_{h\in A(f)}\widehat T(h)=\sup_{h_{u,\alpha}\in A(f)}\widehat T\left(h_{u,\alpha}\right)=
\sup_{h_{u,\alpha}\in A(f)}h_{y(u),\gamma(u,\alpha)},
$$
meaning that
\begin{equation}\label{e24}
T(f)(x)=\sup_{h_{u,\alpha}\in A(f)}\{\langle y(u),x\rangle+\gamma(u,\alpha)\}
=\sup_{h_{u,\alpha}\in A(f)}\{\langle Du+w,x\rangle+\langle d,u\rangle+\tau\alpha+\beta\},
\end{equation}
using Corollary \ref{c5} in the last equality. Define now $C\in$ Aut$(X^*)$, $c\in X^{**}$
as $C=\tau^{-1}D, c=\tau^{-1}d$. 
Continuing from \eqref{e24},
$$
T(f)(x)=
\sup_{h_{u,\alpha}\in A(f)}\{\langle Du,x\rangle +
\langle d,u\rangle+\tau\alpha\}+
\langle w,x\rangle +\beta=
$$
\begin{equation}\label{e25}
\tau\left[\sup_{h_{u,\alpha}\in A(f)}\{\langle Cu,x\rangle +
\langle c,u\rangle+\alpha\}\right]+
\langle w,x\rangle +\beta=
\tau\left[\sup_{h_{u,\alpha}\in A(f)}\{\langle C^*x+c,u\rangle 
+\alpha\}\right]+
\langle w,x\rangle +\beta.
\end{equation}
Note that $C^*$ is an operator in $X^{**}$, so that the expression $C^*x$ in \eqref{e25}
must be understood through the natural immersion of $X$ in $X^{**}$ (i.e., from now on
we consider $X$ as a subspace of $X^{**}$). 
Observe that, by Proposition \ref{p8}(ii),
\begin{equation}\label{ee25}
\sup_{h_{u,\alpha}\in A(f)}\{\langle a,u\rangle +\alpha\}=f^{**}(a).
\end{equation}
Replacing \eqref{ee25} in \eqref{e25}, 
with $a=C^*x+c\in X^{**}$, we get 
\begin{equation}\label{ea25}
T(f)(x) =\tau f^{**}(C^*x+c) +\langle w,x\rangle+\beta.
\end{equation}

Define $Y\subset X^{**}$ as $Y=\{C^*x+c: x\in X\}$. We claim that $Y=X$. Assume first
that there exists $\widetilde x\in X\setminus Y$, and define  
$f_1:X\to\re\cup\{+\infty\}$ as the indicator function
of $\{\widetilde x\}$, i.e. $f_1(x)=0$ if $x=\widetilde x$, $f_1(x)=+\infty$ 
otherwise. It is easy to check that $f_1^{**}:X^{**}\to\re\cup\{+\infty\}$ 
is still the indicator function
of $\{\widetilde x\}$, seen now as a subset of $X^{**}$. We look now at
$T(f_1)$. Since $\widetilde x\notin Y$, we have $C^*x+c\neq\widetilde x$ for
all $x\in X$, so that $f_1^{**}(C^*x+c)=+\infty$ for all $x\in X$. In view
of \eqref{ea25}, $T(f_1)(x)=+\infty$ for all $x\in X$, implying that
$T(f_1)$ is not proper, and so $T(f_1)\notin\mathscr{C}(X)$, contradicting
our assumptions on $T$. We have shown that $X\subset Y$. Suppose now
that there exists $\check x\in Y\setminus X$. Since $\check x$
belongs to $Y$, there exists $x'\in X$ such that $C^*x'+c=\check x$. Define 
$f_2:X\to\re\cup\{+\infty\}$ as the indicator function of $\{x'\}$,
i.e. $f_2(x)=0$ if $x=x'$, $f_2(x)=+\infty$ otherwise. Since $T$ 
is onto, there exists $f\in\mathscr{C}(X)$ such that $T(f)=f_2$. 
Let $\xi=-\tau^{-1}(\langle w,x'\rangle+\beta)$. In view
of \eqref{ea25} we have $f^{**}(C^*x+c)=\xi$ if $x=x'$,
$f^{**}(C^*x+c)=+\infty$ 
otherwise. So, there exists a unique point in $Y$, namely $C^*x'+c=\check x$,
where $f^{**}$ takes a finite value, or equivalently, $f^{**}(z)=+\infty$
for all $z\in Y\setminus\{\check x\}$. Since we already know that $X\subset Y$
and $\check x\notin X$, so that $X\subset Y\setminus\{\check x\}$, 
we conclude  that $f^{**}(z) =+\infty$ for all $z\in X$, and hence, in view
of Proposition \ref{p8}(i), $f(x)=+\infty$ for all $x\in X$, so that $f$ is
not proper, implying that 
$f\notin\mathscr{C}(X)$, a contradiction. We have completed the proof of
the equality between $X$ and $Y$, establishing the claim. Now, since $C^*x+c\in X$
for all $x\in X$, as we have just proved, we apply Proposition \ref{p8}(i),
and rewrite \eqref{ea25} as:
\begin{equation}\label{eb25}
T(f)(x) =\tau f(C^*x+c) +\langle w,x\rangle+\beta.
\end{equation}
We observe now that, since $C^*x+c\in X$ for all $x\in X$, taking $x=0$,
we get $c\in X$. As a consequence, the restriction of $C^*$ to $X$ is
an automorphism of $X$. At this point, that fact that this automorphism
is the restriction to $X$ of the adjoint of an automorphism $C$ of $X^*$,
is not significant any more, since any automorphism $E$ of $X$ is of
the restriction to $X$ of some $C^*:X^{**}\to X^{**}$: it
suffices to define $C:X^*\to X^*$ as $C=E^*$.
Substituting $E$ for $C^*$ in \eqref{eb25}, we recover \eqref{e23},
completing the proof of the ``only if" statement.

Now we prove the ``if" statement, i.e., we must show that operators $T$
of the form given by \eqref{e23} belong to $\mathscr{B}$. First, it is immediate that
such operators map 
$\mathscr{C}(X)$ to $\mathscr{C}(X)$. Now,
we will denote as $T[E,c,w,\tau,\beta]$
the operator $T$ defined by \eqref{e23}. Clearly, $T[E,c,w,\tau,\beta]$ is order preserving.
In order to complete
the proof, we must show that $T[E,c,w,\tau,\beta]$ is fully order preserving and onto.
It suffices to show that $T[E,c,w,\tau,\beta]$ is invertible, and that its inverse is
also order preserving. It is easy to prove that  
\begin{equation}\label{ee26}
T[E,c,w,\tau,\beta]^{-1}=
T[\bar E,\bar c, \bar w,\bar\tau,\bar\beta],
\end{equation}
with
\begin{equation}\label{e27}
\bar E=E^{-1},\quad\bar c=-E^{-1}c,\quad\bar w=-\tau^{-1}\left(E^{-1}\right)^*w,\quad\bar\tau=\tau^{-1},\quad
\bar\beta=\tau^{-1}(\langle E^{-1}c,w\rangle-\beta).
\end{equation}
Note that $E^{-1}$ exists because $E$ is a continuous automorphism of $X$, and is continuous
by virtue of the Closed Graph Theorem (see, e.g., Corollary 2.7 in \cite{Bre}).
Since 
$\bar T= T[\bar C, \bar c, \bar w, \bar\tau, \bar\beta]$ is also of the form given
by \eqref{e23}, we conclude that $T^{-1}=\bar T$ is order preserving, showing that
$T$ is fully order preserving, and completing the proof.
\end{proof}

We  remark that in the reflexive case, the fact that $X=Y$ follows immediately
from the equality between $X$ and $X^{**}$.

The result in Theorem \ref{t1} can be rephrased as saying that 
the identity operator is the only 
fully order preserving operator in $\mathscr{C}(X)$,
up to addition of affine functionals, pre-composition
with affine operators, and multiplication by positive scalars, thus extending
to Banach spaces the result established in \cite{AM2}
for the finite dimensional case.

It is also worthwhile to consider another set of operators on $\mathscr{C}(X)$,  
namely the one consisting
of {\it order preserving involutions}, i.e., of those order preserving operators 
$T:\mathscr{C}(X)\to\mathscr{C}(X)$ which are {\it involutions}, namely operators $T$ 
such that $T(T(f))=f$
for all $f\in \mathscr{C}(X)$. Since for such a $T$ we have $T^{-1}=T$, it follows
that $T^{-1}$ is also order preserving, and hence $T$ is fully order preserving,
so that such set is indeed a subset of $\mathscr{B}$. The characterization of order
preserving involutions in $\mathscr{C}(X)$ is given in the following corollary,
whose finite dimensional version can be found in Theorem 1 of \cite{AM2}. 

\begin{Cor}\label{c7}
An operator $T:\mathscr{C}(X)\to\mathscr{C}(X)$ is an order preserving involution 
if and only if there exist  
$c\in X$, $w\in X^*$ and a  continuous automorphism 
$E$ of $X$, satisfying $E^2=I_{X}, c\in{\rm Ker}(E+I_X), w\in {\rm Ker}(E^*+I_{X^*})$
such that 
\begin{equation}\label{e31}
T(f)(x) =f(Ex+c) +\langle w,x\rangle-\frac{1}{2}\langle c,w\rangle
\end{equation}
for all $x\in X$,
where $I_{X}$ denotes the identity operator in $X$. 
\end{Cor}

\begin{proof}
It is immediate that $T$ is an order preserving involution if and only if $T\in
\mathscr{B}$ and $T=T^{-1}$. In view of the formula for $T^{-1}$ given by \eqref{ee26},
an operator $T\in \mathscr{B}$ satisfies $T=T^{-1}$ if and only if 
$T[E,c,w,\tau,\beta]=T[\bar E,\bar c, \bar w,\bar\tau,\bar\beta]$, with 
$\bar E,\bar c, \bar w,\bar\tau,\bar\beta$  given by \eqref{e27}, which occurs
if and only if $\bar E=E, \bar c =c, \bar w=w, \bar\tau=\tau$ and $\bar\beta=\beta$,
as can be seen after some elementary algebra.
It is easy to check, using \eqref{e27}, that these equalities 
are equivalent to $E^2=I_{X},
c\in$ Ker$(E+I_X), w\in$ Ker$(E^*+I_{X^*})$, $\tau=1$ 
and $\beta=-\langle w,c\rangle/2$,
in which case \eqref{e23} reduces to \eqref{e31}, so that the result is just a 
consequence of Theorem \ref{t1}.
\end{proof}

\section{The order reversing case}\label{s3}

In this section we will characterize fully order reversing operators, but
before defining them formally, we must discuss the appropriate co-domains
for such operators. In \cite{AM2}, the prototypical fully order reversing
operator turns out to be the Fenchel conjugation, and it is reasonable to expect
that it will play a similar role in the infinite
dimensional case. Of course, the operator
which sends $f\in\mathscr{C}(X)$ to its Fenchel conjugate 
$f^*:X^*\to\re\cup\{+\infty\}$ is well
defined in any Banach space; however its co-domain is not  
$\mathscr{C}(X)$, but rather $\mathscr{C}(X^*)$.
There is an additional complication. As we have seen in the order preserving
case, the fact that we are dealing with surjective operators, whose inverses
enjoy a similar order preserving property, is rather essential for the 
characterization. It turns out to be the case that in nonreflexive Banach
spaces the Fenchel conjugation is not onto, generally speaking. In fact,
it is well known that the image
of $\mathscr{C}(X)$ through the Fenchel conjugation coincides with the
class of weak$^*$ lower semicontinuous proper convex functions defined on $X^*$.  

Thus, we define $\mathscr{C}_{w^*}(X^*)$ as the set of weak$^*$ lower semicontinuous
proper convex functions $g:X^*\to\re\cup\{+\infty\}$. We give next the formal definition 
of order reversing operators.
\begin{Def}\label{d2}
\begin{itemize}
\item[i)]  
An operator
$S:\mathscr{C}(X)\to\mathscr{C}_{w^*}(X^*)$ is {\it order reversing} whenever
$f\le g$ implies $S(f) \ge S(g)$.
\item[ii)]
An operator
$S:\mathscr{C}(X)\to\mathscr{C}_{w^*}(X^*)$ is {\it fully order reversing} whenever
$f\le g$ iff $S(f) \ge S(g)$, and additionally $S$ is onto.
\end{itemize}
\end{Def}

We will characterize fully order reversing operators with this definition,
and the result will be an easy consequence of Theorem \ref{t1}.
The issue of characterizing fully order reversing operators $S:\mathscr{C}(X)\to
\mathscr{C}(X)$ (or S:\,\,$\mathscr{C}(X)\to\mathscr{C}(X^*)$ in the nonreflexive
case), are left as open problems, deserving future research.  

We start with an elementary result on the conjugate of compositions with 
affine operators and additions of affine functionals.

\begin{proposition}\label{p9} 
Consider $f\in\mathscr{C}(X)$. Let $E$ be a continuous automorphism of $X$,  
and take $c \in X ,w\in X^*$, $\beta\in\re$, $\tau\in\re_{++}$. Define $g\in\mathscr{C}(X)$ as
$g(x)=\tau f(Ex+c)+\langle w,x\rangle +\beta$. Then
$$
g^*(u)=\tau f^*(H^*u+v)+\langle u,y\rangle+\rho
$$ 
for all $u\in X^*$, where
$$
H=\tau^{-1}E^{-1},\quad v=-\tau^{-1}(E^{-1})^*w,\quad 
y=-\tau^{-1}E^{-1}c\in X,\quad
\rho=\tau^{-1}(\langle w,E^{-1}c\rangle-\beta).
$$
\end{proposition}

\begin{proof}
Elementary.
\end{proof}

Now we present our characterization of fully order reversing operators, which extends
to Banach spaces the similar finite dimensional result in \cite{AM2}.

\begin{Thm}\label{t2}
An operator $S:\mathscr{C}(X)\to\mathscr{C}_{w^*}(X^*)$ is fully order 
reversing
if and only if there exist
$v\in X^*$, $y\in X$, $\rho\in\re$,  $\tau\in\re_{++}$ and a continuous 
automorphism
$H$ of $X$ such that
$$
S(f)(u) =\tau f^*(H^*u+v) +\langle u,y\rangle+\rho,
$$
for all $f\in\mathscr{C}(X)$ and all $u\in X^*$.
\end{Thm}
\begin{proof}
Define $F:\mathscr{C}(X)\to\mathscr{C}(X^*)$ as $F(f)=f^*$, i.e. $F$ is the Fenchel conjugation.
It is immediate from its definition that $F$ is order reversing. Note that
$f^*$ is convex, proper, and also lower semicontinuous in the weak$^*$ topology
for all $f\in\mathscr{C}(X)$, being the supremum of affine functions which are
continuous in this topology. Since $H^*$ is  weak$^*$-continuous (see e.g., \cite{Bre}, p. 81), 
the same holds
for $S(f)$. Hence,  
$F(\mathscr{C}(X))\subset\mathscr{C}_{w^*}(X^*)$.
We recall also that the topological dual of $X^*$ with the weak$^*$ topology is precisely
$X$ with the strong topology (see, e.g., Proposition 2.3(ii) in \cite{BaP}).
We observe that $F$ is invertible; in fact, its inverse 
$F^{-1}:\mathscr{C}_{w^*}(X^*)\to\mathscr{C}(X)$ is given by
\begin{equation}\label{e32}
F^{-1}(g)=\sup_{u\in X^*}\{\langle u,\cdot\rangle-g(u)\}
\end{equation}
for all $g\in\mathscr{C}_{w^*}(X^*)$. 
We conclude from \eqref{e32} that    
$F(\mathscr{C}(X))=\mathscr{C}_{w^*}(X^*)$.
It follows also from \eqref{e32} that $F^{-1}$ is order
reversing, and, since $F$ is also order reversing, we obtain from the  
characterization of the image of $F$ that $F^{-1}$ is fully order reversing.

Consider now the operator $F^{-1}\circ S:\mathscr{C}(X)\to\mathscr{C}(X)$.  Being the composition
of two fully order reversing operators, it is fully order preserving. It follows from
Theorem \ref{t1} that there exist
$c\in X, w\in X^*, \beta\in\re, \tau\in\re_{++}$ and a continuous automorphism $E$ of $X$, such that  
\begin{equation}\label{e33}
F^{-1}\circ S=T[E,c,w,\tau,\beta], 
\end{equation}
where $T[E,c,w,\tau,\beta]\in\mathscr{B}$, as in the proof of Theorem \ref{t1}, 
is defined as
$$
T[E,c,w,\tau,\beta](f)(x)=\tau f(Ex+c)+\langle w,x\rangle+\beta,
$$
for all $f\in\mathscr{C}(X)$ and all $x\in X$. It follows from \eqref{e33} that
$S=F\circ T[E,c,w,\tau,\beta]$, so that $S(f)=g^*$,  where $g\in\mathscr{C}(X)$
is defined as $g(x)=\tau 
f(Ex+c)+\langle w,x\rangle+\beta$. The ``only if" statement follows then from Proposition \ref{p9}.
The ``if" statement is a consequence of the facts that the Fenchel conjugation is a fully order
reversing operator and the affine operator $L(u)=H^*u+v$ is invertible. 
\end{proof}

We remark that when $X$ is reflexive the
weak and the weak$^*$ topologies in $X^*$ coincide, so that $\mathscr{C}(X^*)=
\mathscr{C}_{w^*}(X^*)$, and hence in this case Theorem \ref{t2} characterizes
the set of fully order reversing operators $S:\mathscr{C}(X)\to\mathscr{C}(X^*)$.
In fact, it is easy to check that reflexivity of $X$ is equivalent to surjectivity
of the Fenchel conjugation as a map from $\mathscr{C}(X)$ to $\mathscr{C}(X^*)$.  

As in the case of fully order preserving operators, the result of
Theorem \ref{t2} can be rephrased as saying that
the Fenchel conjugation is the only 
fully order reversing operator from $\mathscr{C}(X)$ to $\mathscr{C}_{w^*}(X^*)$,
up to addition of affine functionals, pre-composition
with affine operators, and multiplication by positive scalars.

\bigskip

\noindent {\bf Acknowledgments:} We thank an anonymous referee and the handling editor for
their very pertinent comments and recommendations.

\end{document}